\documentclass[12pt]{article}
\usepackage{amsfonts}
\usepackage{amsfonts}
\usepackage{mathrsfs}
\usepackage{graphicx}
\usepackage{caption2}
\usepackage{amsfonts}
\usepackage[dvipsnames,usenames]{color}
\usepackage{amsmath, amsthm, amsfonts, amssymb}
\usepackage[toc,page]{appendix}

\newtheorem{thm}{Theorem}[section]
\newtheorem{cor}[thm]{Corollary}
\newtheorem{lem}[thm]{Lemma}

\theoremstyle{definition}
\numberwithin{equation}{section}
\theoremstyle{remark} \hsize=7.5truein \vsize=8.6truein

\def\R{{\hbox{\bf R}}}

\def\P{{\hbox{\bf P}}}
\def\E{{\hbox{\bf E}}}

\def\be{{\bf{e}}}
 at 10 true pt

\def\be#1{ \begin{equation}\label{#1} }
\def\bas{\begin{align*}}
\def\eas{\end{align*}}
\def\bi{\begin{itemize}}
\def\ei{\end{itemize}}

\def\emph#1{{\it #1}}
\def\textbf#1{{\bf #1}}

\theoremstyle{plain}

\theoremstyle{remark}

\theoremstyle{definition}

\begin{document}
\title{Local law for eigenvalues of random regular bipartite graphs }

\author{Linh V. Tran \\
Department of Mathematics, University of Washington}
\date{}
\maketitle

\begin{abstract}
In this paper we study the local law for eigenvalues of large random regular bipartite graphs with degree growing arbitrarily fast. We prove that the empirical spectral distribution of the adjacency matrix converges to a scaled down copy of the Marchenko - Pastur distribution on intervals of short length.
\end{abstract}

\section{Introduction}
In this paper we study the model of random regular bipartite graphs. This is a bipartite analogue of the popular random regular graph model, which is a random graph sampled uniformly from the set of all regular graphs on the same set of vertices. We mainly concern about the asymptotic behavior of the spectral of the adjacency matrices of these graphs as the size of the vertex set goes to infnity. The spectra of random regular graphs was showed to have similar behavior as that of the Erdos-Renyi random graphs, which in turns is similar to the Gaussian orthogonal ensemble on both global and local scales \cite{TVW,EKYY1, EKYY2} This phenomenon is an evidence of the universality conjecture in modern random matrix theory, which roughly states that the spectra of random matrix depends less on the distribution of the entries but more on the algebraic structure of the matrix, so similar matrices with different entry distributions could have similar asymptotic spectral properties.\\

In \cite{DJ13} Dumitriu and Johnson studied the convergence of the empirial spectral distribution of random regular bipartite graphs. Their results show that as the degree grows to infinity, there is a strong connection between the adjacency matrix of random regular bipartite graph with Wishart random matrix. This is an interesting analogue to the connection between the random regular graphs and the Wigner random matrices. Due to a limit in their method, their results only hold if the degree grows slower than any power of the number of vertices.\\

The goal of this paper is to prove a extension of Dumitriu and Johnson's result, which allow the degree to grow at any rate. Our method is very different and largely based on the comparison method in \cite{TVW} which deal with a similar problem of random regular graphs.

\section{Preliminaries and main results}
A $(d_L,d_R)$-regular bipartite graph is a bipartite graph on two sets of vertices $L$ and $R$ so that every vetex of $L$ (or $R$) has degree $d_L$ (or $d_R$, respectively). The model $G_{m,n,d_L,d_R}$ is defined as a random graph sampled uniformly from the set of all $(d_L,d_R)-$ regular bipartite graphs on two sets $L$ and $R$ and $|L|=m$, $|R|=n$. The adjacency matrix of $G_{m,n,d_L,d_R}$ is a random matrix $A$ of the following form (under proper labeling of vertices)
\begin{equation}A=\begin{pmatrix}
0&X\\
X^T&0
\end{pmatrix}
\end{equation}
where $X$ is a $m\times n$ $(0,1)$ random matrix. It's easy to show that the non zero eigenvalues of $A$ come in pairs $(-\lambda,\lambda)$ where $\lambda^2$ is an eigenvalue of $X^TX$, and $A$ has at least $m-n$ zero eigenvalues. Also assume that $m$ and $n$ increase to infinity in the way that
$$\frac{d_R}{d_L}=\frac{m}{n} \longrightarrow\alpha\ge1$$
We will compare $G_{m,n,d_L,d_R}$ with the Erdos-Renyi bipartite random graph model $G(m,n,p)$ defined on two sets of vertices $L$, $R$ and each edge from a vertex in $L$ to a vertex in $R$ is chosen randomly and independently with probability $p$. Under proper vertex labeling the adjacency matrix $B$ of $G(m,n,p)$ has the form
\begin{equation}B=\begin{pmatrix}
0&Y\\
Y^T&0
\end{pmatrix}
\end{equation}
where $Y$ is a $m\times n$ random matrix with iid entries (1 with probability $p$, 0 with probability $1-p$).\\

For a $n\times n$ Hermitian matrix $M$ with real eigenvalues $\lambda_1\le\lambda_2\le\dots\le\lambda_n$, the {\it empirical spectral distribution } (ESD) is the probability measure $\mu_n(M)$ defined as
$$\mu_n(M)=\frac{1}{n}\sum_{i=1}^{n}\delta_{\lambda_i}$$
where $\delta_\lambda$ is the Dirac point measure at point $\lambda$. Note that if $M$ is a random matrix then $\mu_n(M)$ is a random measure.\\
 
 It's well known that if $M$ is an $m\times n$ random matrix whose entries are iid copies of a random variable with mean zero and variance one, and $m/n$ converges to a finite limit $\alpha$, then the ESD of $\frac{1}{n}M^TM$ converges to the Marcenko-Pastur distribution $\nu_{MP}$ of ratio $1/\alpha$. Thus it is natural to expect that the $ESD$ of $\frac{1}{d_L}X^TX$ with $X$ from the adjacency matrix of $G_{m,n,d_L,d_R}$ also converges to the Marcenko-Pastur distribution when $d_L$ grows to infinity.\\
 
 Recall that the Marcenko-Pastur distribution with ratio $1/\alpha$ is supported on $[a,b]$ and given by the density function
 $$p(x)=\frac{\alpha}{2\pi x}\sqrt{(b^2-x)(x-a^2)},$$
 where $a=1-\alpha^{-1/2}$ and $b=1+\alpha^{-1/2}$. If the limiting ESD of ${d_L^{-1}}X^TX$ is the Marcenko-Partur law then because the ESD $\mu_n$ of $d_L^{-1/2}A$ is the distribution of the square roots (both positive and negative) of eigenvalues of  ${d_L^{-1}}X^TX$, we can find the limiting for ESD of $d_L^{-1/2}A$ as well. The limiting ESD $\mu$ will have support on $[-b,-a]\cup\{0\}\cup[a,b]$ with density function
 \begin{equation}\label{def.mu}q(x)=\frac{2|x|}{1+\alpha}p(x^2)=\frac{\alpha}{(1+\alpha)\pi|x|}\sqrt{b^2-x^2)(x^2-a^2)},\end{equation}
 and a point mass of $\frac{\alpha-1}{\alpha+1}$ at $0$. Indeed, the $m-n$ zero eigenvalues give the point mass of $\frac{m-n}{m+n}$ at 0, while the other $2n$ eigenvalues are described by applying a change of variable $x$ to $\sqrt{x}$ to $p(x)$ and scaled by factor $\frac{2n}{m+n}$.\\
 
In \cite{DJ13} Dumitriu and Johnson proved that if $d_L=o(n^{\epsilon})$ for a fixed $\epsilon>0$ then the ESD of  $\frac{1}{d_L}X^TX$ converges to the Marcenko-Pastur law. They also proved a ``local law" that if $d_L=\exp(o(1)\sqrt{\log n})$ then for any interval $I$ of short length and not containing 0, it holds that for all $\delta>0$ and $n$ large enough 
 $$|\mu_n(I)-\mu(I)|<\delta C|I|$$
 with probability $1-o(1/n)$, where $\mu_n$ is the ESD of $d_L^{-1/2}A$ and $\mu$ is the limiting distribution defined by (\ref{def.mu}).\\

We are going to prove the same local law for the case $d_L=\omega(\log n)$, which is a complement of the result by Dumitriu and Johnson. Let $R$ be the normalized adjacency matrix of $G_{m,n,d_L,d_R}$
\begin{align} R=\frac{1}{\sqrt{\frac{d_L}{n}(1-\frac{d_L}{n})}}\left[ A-\frac{d_L}{n}\begin{pmatrix}0&J\\J^T&0\end{pmatrix}\right]\end{align}
where $J$ is the all 1 $m\times n$ matrix. Since $\begin{pmatrix}0&J\\J^T&0\end{pmatrix}$ has rank 2, the ESD of $R$ has the same global behavior as that of $A$ due to interlacing principle. Our main result is
\begin{thm}\label{mainthm} Suppose $d_L=\omega(\log n)$ as $n$ tends to infinity. Let $\delta>0$ and $N_I$ be the number of eigenvalues of $n^{-1/2}R$ in the interval $I$, where $I$ is an interval avoiding $\{0\}$ with length at least $\big(\frac{\log d_L}{\delta^3d_L^{1/2}}\big)^{1/4}$, then 
$$|N_I-n\mu(I)|<\delta n\mu(I)$$
with probability at least $1-O(\exp(-cnd_L\log(d_L))$.
\end{thm} 
The convergence of ESD of $A$ is a direct consequence of Theorem \ref{mainthm}.
\begin{cor} The ESD of $d_L^{-1/2}A$ converges in distribution to the limting measure $\mu$ (as defined by (\ref{def.mu}))
\end{cor}

The rest of the paper is organized as follow: In Section 3 we calculate the probability for Erdos-Renyi random bipartite graphs to be regular. This result provides a tool to compare the Erdos-Renyi model with the regular model. In section 4 we prove Theorem \ref{mainthm} via a more general concentration result about Wishart-like random matrix.\\

{\bf Acknowledgement.} The author thanks I. Dumitriu for bringing the problem to his attention.

\section{Probability of random bipartite graphs to be regular}
We will prove a lower bound for the probability of Erdos-Renyi random bipartite graphs to be regular. Recall that the Erdos-Renyi bipartite graphs model $G(m,n,p)$ consists of two vertex sets $A$ and $B$ whose capacities are $m$ and $n$ respectively, and an edge between a vertex of $A$ and a vertex of $B$ is chosen randomly and independently with probability $p$. 
\begin{lem}\label{lem1}  If $np=\Omega(\log n)$ and $m/n\rightarrow\alpha<\infty$ as $n\rightarrow\infty$ then $G(m,n,p)$ is $(mp,np)$-regular with probability at least $\exp(-O(n(np)^{1/2})$.
\end{lem}

We will employ the following criteria for a bipartite graph $G$ to contain a $f$ - factor. Let $f$ be an integer-valued function on the vertices of $G$ so that $f(v)\le \deg(v)$ for any $v$. A $f$ - factor is a subgraph with $f$ to be its degree sequence.
\begin{thm}[Ore-Ryser]\label{OR-thm} Let $A$ and $B$ be the two vertex partitions of $G$. For $S$ be a subset of $A$ and $v\in B$, $d_S(v)$ denotes the number of neighbors of $v$ in $S$. Then the two following statements are equivalent:\\
(i) $G$ contains a $f$ - factor \\
(ii) Every subset $S$ of $A$ satisfies
$$ \sum_{v\in B}\min(f(v),d_S(v))\ge\sum_{u\in S}f(u).$$
\end{thm}
In particular, if $f(v)=x$ for all $v\in A$ and $f(v)=y$ for all $v\in B$ then we call the $f$-factor the $(x,y)$-regular factor, and the condition $(ii)$ becomes 
$$ \sum_{v\in B}\min(y,d_S(v))\ge |S|x.$$

Denote $d_L=np$, $d_R=mp$. Following the argument of Shamir and Upfal \cite[Section 5]{SU}, Lemma \ref{lem1} is a direct consequence of the following analogue of Theorem 1 in \cite{SU}
\begin{lem}\label{lem2}  Let $d_L=\omega\log n$ and $\delta=\omega^{-\theta}$ where $\theta<1/2$ and $\omega=\omega(n)\rightarrow\infty$ arbitrarily slowly as $n$ goes to infinity. Let $d_L'=d_L(1-\delta)$, $d_R'=d_R(1-\delta)$. Then the Erdos-Renyi random bipartite graph $G(m,n,p)$ contains a $(d_L',d_R')$-regular factor with probability $1-O(n^{-\omega^{1-2\theta}})$.
\end{lem}
\begin{proof}
Let $S$ be a subset of the vertex set $A$ of $G(m,n,p)$ and $|S|=k$ (so $1\le k\le m$). $S$ is called ``bad" if it doesn't satisfy the condition $(ii)$ of Theorem \ref{OR-thm}, i.e. 
$$Y=\sum_{v\in B}\min(d_S(v),d_R')<kd_L'.$$
If $d_S(v)<d'_R$ for all $v$ in $B$ then $Y$ will be just $d_S(B)$, the number of edges from $B$ to $S$. Since $\E(d_S(v))=kd_R/m$, by the Chernoff bound we can see that
$$\P(d_S(v)>d'_R)\le \exp(-O(\frac{m}{k}d_R)).$$
Let $X_S$ be the number of $v$ in $B$ such that $d_S(v)>d'_R$, then $\E X\le ne^{-O(\frac{m}{k}d_R)}=n^{-O(\frac{m}{k}\omega)}$. Again by Chernoff bound, if $C$ is an absolute constant then
$$\P(X_S>C)\le \exp(-n^{O(\frac{m}{k}\omega)}).$$
If $X_S\le C$ then $Y$ will lose at most $Ck$ and
\begin{align*}
\P(Y<kd_L(1-\delta))&\le\P(d_S(B)<k(d_L(1-\delta)+C))\\
&\le \exp(-kd_L\omega^{-2\theta})
\end{align*}
Therefore
\begin{align*}
\P(S \text{ is bad})&\le \P(\{S \text{ is bad} \}\wedge\{X_S\le C\})+\P(X_S>C)\\
&\le \P(\sum_{v\in B}d_S(v)<kd_L')+n\P(d_S(v)>d_R')\\
&\le \exp(-k\omega^{1-2\theta}\log n ) 
\end{align*}
By union bound, the probability that there is a bad set $S$ is then at most $O(n^{-\omega^{1-2\theta}})$.
\end{proof}

\section{Proof of Theorem \ref{mainthm}} We use the comparison method.   
 A key ingredient of the proof is the following concentration lemma, which may be of independent interest. 

\begin{lem}\label{thm:ESD-con-general} Let $M$ be a $(m+n)\times (m+n)$ Hermitian random matrix of the form
\begin{equation*}
M_n=
\begin{pmatrix}
0&X\\
X^{T}&0
\end{pmatrix}
\end{equation*}

 where $X$ be $m\times n$ random matrix whose entries $\xi_{ij}$ are i.i.d. random variables with mean zero, variance 1 and $|\xi_{ij}|< K$ for some common constant $K$. Fix $\delta>0$ and assume that the eighth moment  $M_8:=\sup_{i,j}\E(|\xi_{ij}|^8)<\infty$. Then for any interval $I\subset \R$ avoiding $\{0\}$  whose length is at least $\Omega(\delta^{-1/2}(M_8^{1/8}n^{-1/4})$, there is a constant $c>0$ such that the number $N_I$ of the eigenvalues of $\frac{1}{\sqrt{n}}M$ which belong to $I$ satisfies the following concentration inequality
$$\P(|N_I-n\mu(I)|>\delta n\mu(I))\le 4\exp(-c\frac{\delta^3n^2|I|^4}{K^2}),$$
where $\mu$ is the limiting distribution defined by (\ref{def.mu}).
\end{lem}

Apply Lemma  \ref{thm:ESD-con-general} for the normalized adjacency matrix  of $G(m,n,p)$ 
$$M=\frac{1}{\sqrt{p(1-p)}}\left[B-p\begin{pmatrix}0&J\\J^T&0\end{pmatrix}\right]$$
with $K=1/\sqrt{p}$  we obtain 

\begin{cor}\label{cor:ESDconvrate} Let $\delta>0$ and $N_I$ be the number of eigenvalues of $M$ inside interval $I$ avoiding $\{0\}$ with length at least $\big(\frac{\log(np)}{\delta^3(np)^{1/2}}\big)^{1/4}$, there is a constant $c>0$ so that
$$|N_I-n\mu(I)|\ge \delta n\mu(I)$$
with probability at most $\exp(-cn(np)^{1/2}\log(np))$.
\end{cor}

By Corollary \ref{cor:ESDconvrate} and Lemma \ref{lem1}, the probability that $N_I$ fails to be close to the expected value in the model $G(m,n,p)$ is much smaller than the probability that $G(m,n,p)$ is $(mp,np)$-regular. Thus the probability that $N_I$ fails to be close to the expected value in the model $G_{m,n,d_L,d_R}$ where $d_L=np$, $d_R=mp$ is the ratio of the two former probabilities, which is $O(\exp(-cn\sqrt{np}\log np))$ for some small positive constant $c$. Thus, Theorem \ref{mainthm} is proved, depending on Lemma \ref{thm:ESD-con-general} which we turn to next. 

\subsection
{Proof of Lemma  \ref{thm:ESD-con-general}} 

Assume $I=[a,b]$ where $a<b<0$ or $0<a<b$. 

We will use the approach of Guionnet and Zeitouni in \cite{GZ}. Consider a random Hermitian matrix $W_n$ with independent entries $(W_n)_{ij}=A_{ij}w_{ij}$ where
\begin{itemize}
\item $A=(A_{ij})$ is a deterministic matrix of the form
\begin{equation*}A=
\begin{pmatrix}
0&J\\
J^{T}&0
\end{pmatrix}
\end{equation*}
with $J$ be the $m\times n$ all 1 matrix.
\item $w_{ij}$'s are iid copies of a random variable $w$ with mean zero, variance one, support in a compact region $S$. Moreover $w$ is bounded by a constant $K$. 
\end{itemize}

Let $f$ be a real convex $L$-Lipschitz function and define
$$Z:=\sum_{i=1}^nf(\lambda_i)$$
where $\lambda_i$'s are the eigenvalues of $\frac{1}{\sqrt{n}}W_n$. We are going to view $Z$ as the function of the variables $w_{ij}$. For our application we need $w_{ij}$ to be random variables with mean zero and variance 1, whose absolute values are bounded by a common constant $K$ ($K$ may depend on $n$).

The following concentration inequality is a version of Theorem 1.1 in \cite{GZ}.

\begin{lem}\label{GZconcentration} Let $W_n, f, Z$ be as above. Then there is a constant $c>0$ such that for any $T>0$
$$\P(|Z-\E(Z)|\ge T)\le 4\exp(-c\frac{T^2}{K^2L^2}).$$
\end{lem}

In order to apply Lemma \ref{GZconcentration} for $N_I$ and $M$, it is natural to consider 
$$Z:=N_I=\sum_{i=1}^n \chi_I(\lambda_i)$$
where $\chi_I$ is the indicator function of $I$ and $\lambda_i$ are the eigenvalues of $\frac{1}{\sqrt{n}}M_n$. However, this function is neither convex nor Lipschitz. As suggested in \cite{GZ}, one can  overcome this problem by a proper approximation.  Define $I_l=[a-\frac{|I|}{C},a]$, $I_r=[b,b+\frac{|I|}{C}]$, where $C$ is a constant to be chosen later, and construct two real functions $f_1, f_2$ as follows(see Figure \ref{fig:fg}): 
\begin{equation*}
f_1(x)=\Bigg\{
\begin{array}{ll}
-\frac{C}{|I|}(x-a)-1& \text{if }x\in (-\infty, a-\frac{|I|}{C})\\
0&\text{if }x\in I\cup I_l\cup I_r\\
\frac{C}{|I|}(x-b)-1& \text{if }x\in (b+\frac{|I|}{C},\infty)
\end{array}
\end{equation*}

\begin{equation*}
f_2(x)=\Bigg\{
\begin{array}{ll}
-\frac{C}{|I|}(x-a)-1& \text{if }x\in (-\infty, a)\\
-1&\text{if }x\in I\\
\frac{C}{|I|}(x-b)-1& \text{if }x\in (b,\infty)
\end{array}
\end{equation*}

\begin{figure} [htbp]
  \centering 
  \includegraphics[scale=0.52]{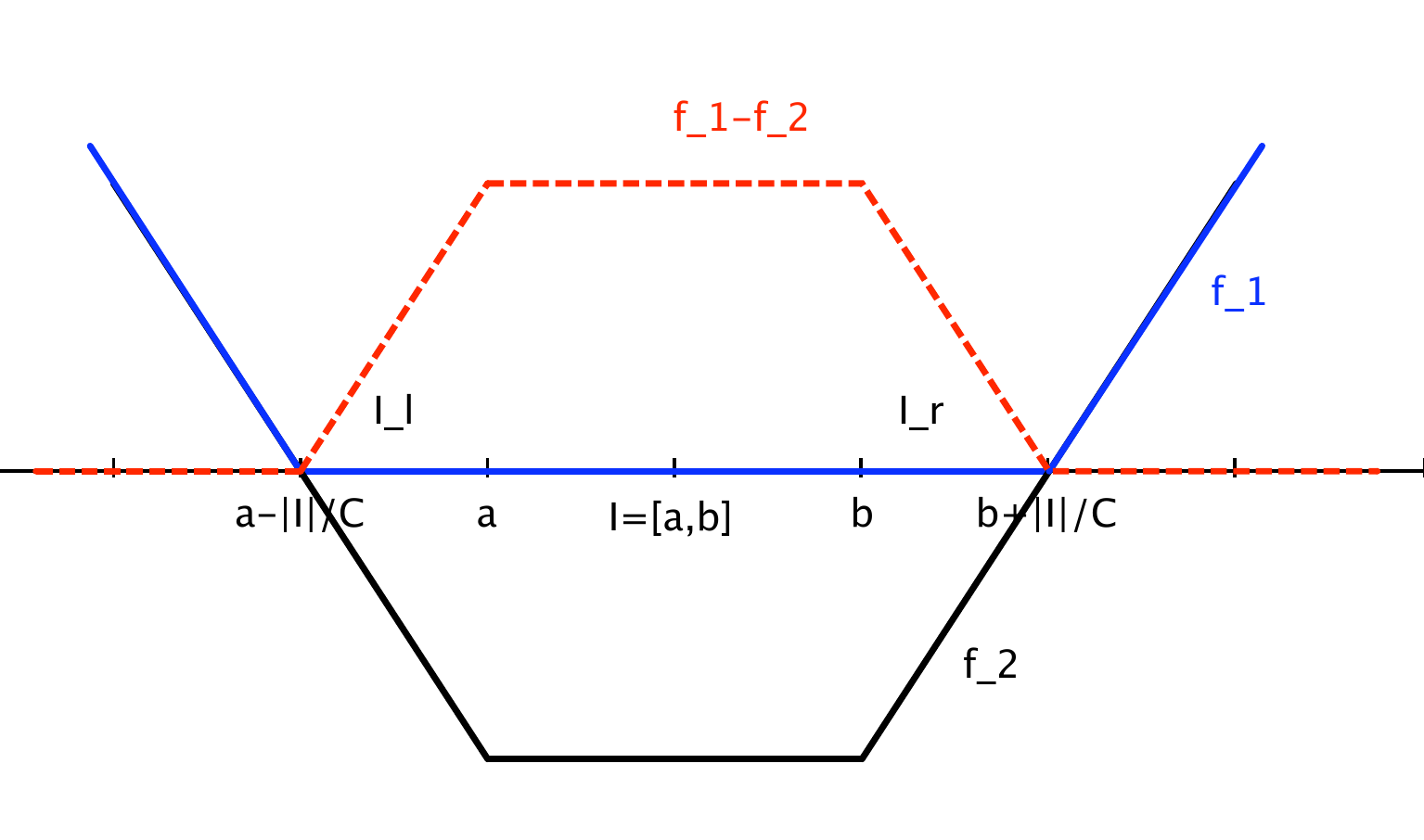}%
  \hfill{}%
  \includegraphics[scale=0.52]{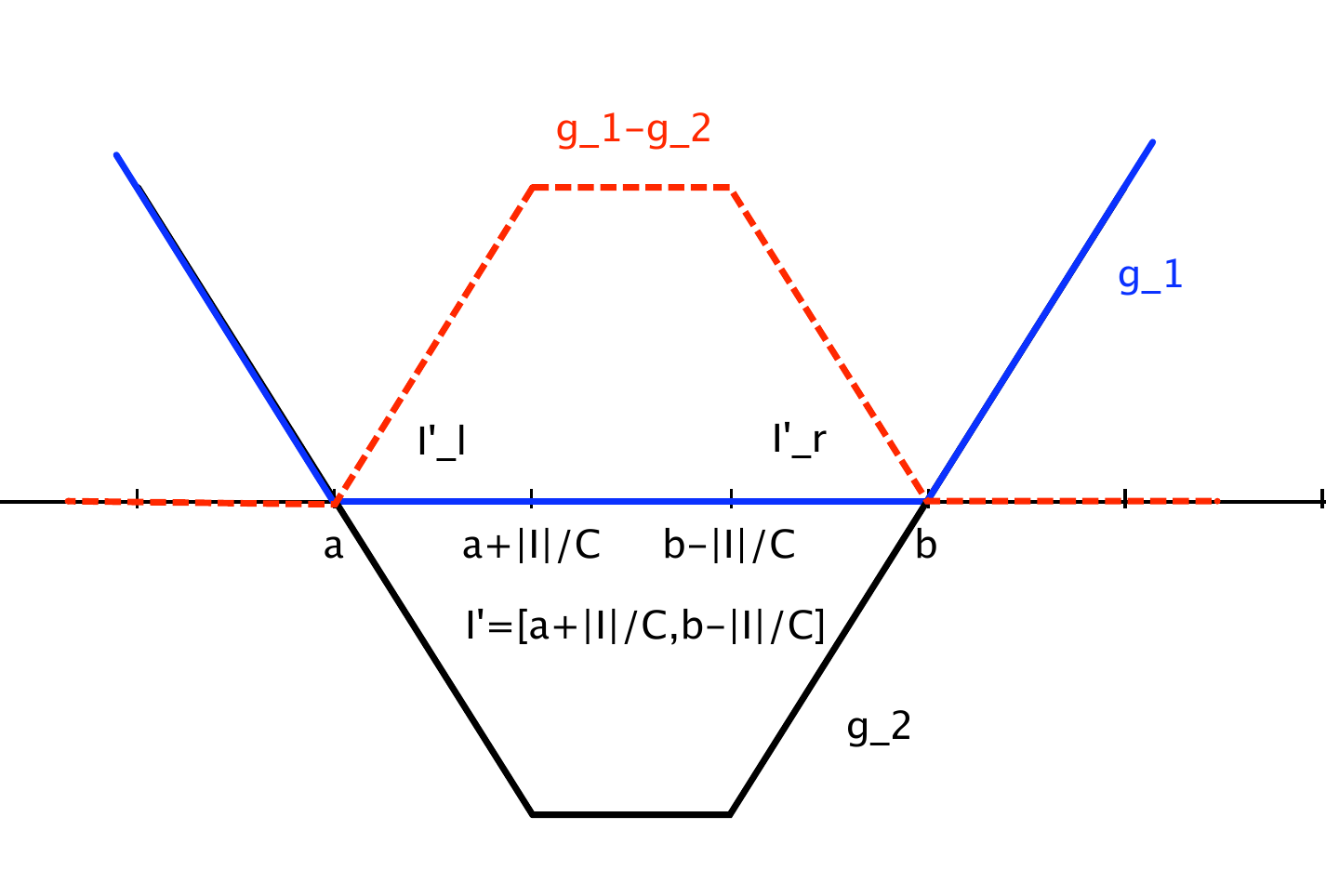} 
  \caption{Auxiliary functions used in the proof} 
  \label{fig:fg}
\end{figure}

Note that $f_j$'s are convex and $\frac{C}{|I|}$-Lipschitz. Define $$X_1=\sum_{i=1}^n f_1(\lambda_i),\ X_2=\sum_{i=1}^n f_2(\lambda_i)$$ and  apply Lemma \ref{GZconcentration} with $T=\frac{\delta}{8} n\mu(I)$ for $X_1$ and $X_2$. Thus, we have
\begin{align*}
\P(|X_j-\E(X_j)|\ge \frac{\delta}{8} n\mu(I))&\le4\exp(-c\frac{\delta^2n^2|I|^2(\mu(I))^2}{K^2C^2}).\\
\end{align*}
Direct calculation shows that for $I$ in the support of $\mu$ one have $\mu(I)\le\alpha|I|^2$ for some absolute constant $\alpha$. Thus we have for $j=1,2$
$$\P(|X_j-\E(X_j)|\ge \frac{\delta}{8} n\mu(I))\le4\exp(-c_1\frac{\delta^2n^2|I|^4}{K^2C^2})$$

 Let $X=X_1-X_2$, then 
$$\P(|X-\E(X)|\ge \frac{\delta}{4} n\mu(I))\le O(\exp(-c_1\frac{\delta^2n^2|I|^4}{K^2C^2})).$$

Now we compare $X$ to $Z$, making use of the following result about convergence rate for Marchenko - Pastur law by G\"otze and Tikhomirov .

\begin{lem}[\cite{GT-MP}Theorem 1.1]\label{lemGT} Let $W_n=(\omega_{ij})$ be a $m\times n$ random matrices whose entries are independent with mean zero and variance one, and $M_8=\sup_{i,j}\E(|\omega_{ij}|^8)<\infty$. Then for any $I\subset\R$ the number $N'_I$ of eigenvalues of $\frac{1}{\sqrt{n}}W_n^{T}W_n$ inside $I$ satisfies
$$|\E(N'_I)-n\mu_{MP}(I)|<\beta' n\frac{M_8^{1/4}}{\sqrt{n}},$$
where $\beta'$ is an absolute constant.
\end{lem}
Since $\mu$ is a scaled down copy of $\mu_MP$, the same convergence rate (with another constant) holds for our case
$$|\E(N_I)-n\mu(I)|<\beta M_8^{1/4}n^{1/2},$$

We have $\E(X-Z)\le \E(N_{I_l}+N_{I_r})$. Thus by Lemma \ref{lemGT}
$$\E(X)\le\E(Z)+n(\mu(I_l)+\mu(I_r))+\beta M_8^{1/4}n^{1/2}.$$
Choose $C=(4/\delta)^{1/2}$, then because $|I|\ge \Omega(\delta^{-1/2}(M_8^{1/8}n^{-1/4})$,
$$n(\mu(l_l)+\mu(I_r))=\Theta(n(\frac{|I|}{C})^{2})>\Omega( M_8^{1/4}n^{-1/2})$$
and
$$n(\mu(I_l)+\mu(I_r))+\beta M_8^{1/4}n^{1/2}=\Theta(n(\frac{|I|}{C})^{2})=\Theta(\frac{\delta}{4}n\mu(I)).$$

Therefore, with probability at least $1-O(\exp(-c_1\frac{\delta^4n^2|I|^4}{K^2}))$, we have
$$Z\le X\le \E(X)+ \frac{\delta}{4}n\mu(I) < \E(Z)+\frac{\delta}{2} n\mu(I).$$
Lemma \ref{lemGT} again gives
$$\E(N_I)<n\mu(I)+ \beta M_8^{1/4}n^{1/2}<(1+\frac{\delta}{2})n\mu(I),$$
hence with probability at least $1-O(\exp(-c_1\frac{\delta^3n^2|I|^4}{K^2}))$
$$N_I<(1+\delta) n\mu(I),$$
which is the desires upper bound.

The lower bound is proved using a similar argument. Let $I'=[a+\frac{|I|}{C}, b-\frac{|I|}{C}]$, $I'_l=[a,a+\frac{|I|}{C}]$, $I'_r=[b-\frac{|I|}{C},b]$ where $C$ is to be chosen later and define two functions $g_1$, $g_2$ as follows (see Figure \ref{fig:fg}):

\begin{equation*}
g_1(x)=\Bigg\{
\begin{array}{ll}
-\frac{C}{|I|}(x-a)& \text{if }x\in (-\infty, a)\\
0&\text{if }x\in I'\cup I'_l\cup I'_r\\
\frac{C}{|I|}(x-b)& \text{if }x\in (b,\infty)
\end{array}
\end{equation*}

\begin{equation*}
g_2(x)=\Bigg\{
\begin{array}{ll}
-\frac{C}{|I|}(x-a)& \text{if }x\in (-\infty, a+\frac{|I|}{C})\\
-1&\text{if }x\in I'\\
\frac{C}{|I|}(x-b)& \text{if }x\in (b-\frac{|I|}{C},\infty)
\end{array}
\end{equation*}

Define $$Y_1=\sum_{i=1}g_1(\lambda_i),\ Y_2=\sum_{i=1}g_2(\lambda_i).$$ A similar argument using Lemma \ref{GZconcentration}  and Lemma \ref{lemGT} with $Y_1$, $Y_2$ in place of $X_1$, $X_2$ shows that  with probability  at least $1-O(\exp(-c_2\frac{\delta^3n^2|I|^4}{K^2C^2}))$
$$N_I>(1-\delta)n\mu(I).$$
Thus, Lemma \ref{thm:ESD-con-general} is proved.
\endproof


\bibliographystyle{plain}
\bibliography{random-matrices}

\end{document}